\documentclass[11pt]{article}

\usepackage{amssymb,amsmath,euscript,bbm,xcolor,graphicx,epstopdf}
\RequirePackage[numbers]{natbib}

\usepackage{amsmath,amssymb,amsthm,bm,euscript,bbm}
\usepackage{verbatim}
\usepackage{graphicx}
\usepackage{epstopdf}
\usepackage{color}
\usepackage{pstricks,pst-node,pst-plot}
\usepackage{rotating}
\usepackage{enumitem}
\newtheorem{proposition}{Proposition}[section]
\newtheorem{lemma}[proposition]{Lemma}
\newtheorem{theorem}[proposition]{Theorem}

\def\la{\lambda}

\def\ep{\varepsilon}

\def\l{{\langle}}
\def\r{\rangle}

\def\R{{\mathbb R}}
\def\S{{\mathbb S}}

\def\E{{\mathbb E}}
\def\P{{\mathbb P}}

\newcommand{\GOI}{{\rm GOI}}

\makeatletter \@addtoreset{equation}{section} \makeatother
\newenvironment{example}{%
	\vspace{0.3cm} \pagebreak [2]%
	\par%
	\refstepcounter{proposition}%
	\noindent%
	{\bf  Example~\theproposition\ }}{}%
\begin{document}
	
	\title {Smooth Mat\'ern Gaussian Random Fields: Euler Characteristic, Expected Number and Height Distribution of Critical Points\thanks{Research partially supported by NSF DMS-2220523 and Simons Foundation Collaboration Grant 854127. } }
	\author{Dan Cheng\\ Arizona State University}
	
	\date{}
	
	\maketitle
	
	\begin{abstract}
		This paper studies Gaussian random fields with Mat\'ern covariance functions with smooth parameter $\nu>2$. Two cases of parameter spaces, the Euclidean space and $N$-dimensional sphere, are considered. For such smooth Gaussian fields, we have derived the explicit formulae for the expected Euler characteristic of the excursion set, the expected number and height distribution of critical points. The results are valuable for approximating the excursion probability in family-wise error control and for computing p-values in peak inference.
	\end{abstract}
	
	
	
	\section{Introduction}
	
	In recent years the Mat\'ern class of covariance functions has gained widespread
	popularity in spatial statistics \cite{Guttorp:2006,Stein:1999}, mainly due to its great flexibility for modelling the spatial covariance and hence dependent spatial data and processes. Specifically, the \textit{Mat\'ern covariance function} is defined as  
	\begin{equation}\label{eq:Matern}
		\mathcal{M}(d) = \frac{\sigma^2}{2^{\nu-1}\Gamma(\nu)}\left( \frac{\sqrt{2\nu}d}{\ell} \right)^\nu K_\nu\left( \frac{\sqrt{2\nu}d}{\ell} \right), \quad d\ge 0,
	\end{equation}
	where $K_\nu(\cdot)$ is the modified Bessel function of the second kind. This covariance function has three positive parameters $\sigma^2$, $\ell$ and $\nu $, with $\sigma^2$ controlling the variance, $1/\ell$ controlling the spatial range, and $\nu$ controlling the smoothness. Let $X=\{X(t), t\in T\}$ be a centered Gaussian random field living on an $N$-dimensional parameter set $T$. We call $X$ a \textit{Mat\'ern Gaussian random field} if the covariance satisfies
	\[
	\E[X(t)X(s)] = \mathcal{M}(\|t-s\|), \quad t,s\in T.
	\]
	That is, $X$ is isotropic with covariance given by \eqref{eq:Matern}. There is rich literature on Mat\'ern Gaussian fields; however, only a few are focused on the smooth case. This paper aims to bridging smooth Mat\'ern Gaussian fields and recent developments on the expected Euler characteristic of the excursion set \cite{AT07,ChengXiao16}, the expected number and height distribution of critical points \cite{AstrophysJ85,CS15,CS17,CS18,ChangePoint}.
	
	Let $\{X(t), t\in T\}$ be a centered smooth Gaussian random field. Let $A_u(X, T)=\{t\in T: X(t) \ge u\}$ be the excursion set of $X$ exceeding level $u$ over the parameter set $T$ and denote by $\chi(A_u(X, T))$ its Euler characteristic. It is shown in \cite{AT07} that, the expected Euler characteristic $\E[\chi(A_u(X, T))]$ is computable and can be used to approximate the excursion probability $\P[\sup_{t\in T} X(t) \ge u]$ for large $u$ such that the error is super-exponentially small. This is useful in controlling the family-wise error in statistics \cite{Taylor:2007}.
	
	The number of critical points of index $i$ of $X$ above $u$ on a unit-area disc is defined as
	\begin{equation}\label{Eq:mu-i}
		\mu_i(X, u) = \# \left\{ t\in D: X(t)\geq u, \nabla X(t)=0, \text{index} (\nabla^2 X(t))=i \right\}, \quad i=0, \ldots, N,
	\end{equation}
	where $D$ is an $N$-dimensional unit-area disc on $T$, $\nabla X(t)$ and $\nabla^2 X(t)$ are respectively the gradient and Hessian of $X$, and $\text{index} (\nabla^2 X(t))$ denotes the number of negative eigenvalues of $\nabla^2 X(t)$. Notice that we will focus on the expectation of $\mu_i(X, u)$, which depends on the volume of the area of interest due to isotropy, thus we omit $D$ in the notation for simplicity. The expected number of critical points $\E[\mu_i(X, u)]$ of smooth random fields is important in statistics \cite{CS18} and physics \cite{AstrophysJ85}. By default, let $\mu_i(X) = \mu_i(X, -\infty)$ be the number of critical points of index $i$ of $X$ over a unit disc.
	
	The height distribution of a critical value of index $i$ of $X$ at $t_0$ is defined as
	\begin{equation}\label{Eq:F}
		F_i(u) :=\lim_{\ep\to 0} \P\left\{X(t_0)>u \ | \ \exists \text{ a critical point of index $i$ of } X(t) \text{ in } B(t_0, \ep) \right\},
	\end{equation}
	where $B(t_0, \ep)$ is the ball of radius $\ep$ centered at $t_0$. It has been found to be an important tool for computing p-values in peak detection and thresholding problems in statistics \cite{CS17,CS19sphere,ChangePoint} and neuroimaging \cite{Chumbley:2009}. Notice that, due to isotropy, $F_i(u)$ does not depend on the location $t_0$. It is shown in \cite{CS15} that,
	\begin{equation}\label{Eq:F-ratio}
		F_i(u) = \frac{\E[\mu_i(X, u)]}{\E[\mu_i(X)]}.
	\end{equation}
	Therefore, $F_i(u)$ can be obtained immediately once the evaluation of $\E[\mu_i(X, u)]$ is known.

	\section{Smooth Mat\'ern Gaussian Random Fields on Euclidean Space}
	In studying the expected Euler characteristic and critical points, it requires the Gaussian fields to be twice differentiable \cite{AT07,CS15}; thus, throughout this paper, we assume $\nu>2$. To characterize the variances of the derivatives of $X$, following the notation in \cite{CS15}, we introduce a real function $\rho(d^2) = \mathcal{M}(d)$, or equivalently $\rho(d) = \mathcal{M}(\sqrt{d})$, and have the following results.
	
	\begin{proposition}\label{prop:1}
		Let $\rho(d) = \mathcal{M}(\sqrt{d})$, and let $\rho'=\rho'(0)$ and $\rho''=\rho''(0)$. Suppose $\nu>2$. Then
		\begin{equation}\label{eq:rho'}
			\begin{split}
				\rho' = -\frac{\sigma^2\nu}{2(\nu-1)\ell^2} \quad  \text{\rm and } \quad 
				\rho'' = \frac{\sigma^2\nu^2}{4(\nu-1)(\nu-2)\ell^4}.
			\end{split}
		\end{equation} 
		In particular,
		\begin{equation}\label{eq:kappa}
			\begin{split}
				\kappa := -\frac{\rho'}{\sqrt{\rho''}}\bigg\rvert_{\sigma=1} = \sqrt{\frac{\nu-2}{\nu-1}}\quad  \text{\rm and } \quad  \eta := \frac{\sqrt{-\rho'}}{\sqrt{\rho''}} = \sqrt{\frac{2(\nu-2)}{\nu}}\ell.
			\end{split}
		\end{equation} 
	\end{proposition}
	\begin{proof} It follows from the property of Bessel functions (see page 502 in \cite{Spanier:1987}) that, for a non-integer $\nu$,
		\begin{equation}\label{eq:K}
			\begin{split}
				K_\nu(r) = \frac{2^{\nu-1}\Gamma(\nu)}{r^\nu}\sum_{j=0}^\infty \frac{(r^2/4)^j}{j!(1-\nu)_j} + \frac{r^\nu\Gamma(-\nu)}{2^{\nu+1}}\sum_{j=0}^\infty \frac{(r^2/4)^j}{j!(1+\nu)_j},
			\end{split}
		\end{equation}  
	where $(x)_j=x(x+1)\cdots(x+j-1)$ with $(x)_0=1$ is the Pochhammer symbol. For $\nu>2$, organizing the right side of \eqref{eq:K}, we obtain the following expansion as $r\to 0$,
	\begin{equation}\label{eq:Taylor1}
		\begin{split}
			\frac{r^\nu}{2^{\nu-1}\Gamma(\nu)}K_\nu(r) = 1 - \frac{r^2}{4(\nu-1)} + \frac{r^4}{32(\nu-1)(\nu-2)} + o(r^4).
		\end{split}
	\end{equation}  
	 On the other hand, by page 502 in \cite{Spanier:1987}, if $\nu=n$ is an integer, then
	\begin{equation*}
		\begin{split}
			K_n(r) = &\frac{2^{n-1}}{r^n}\sum_{j=0}^{n-1} \frac{(-1)^j(n-j-1)!(r^2/4)^j}{j!} \\
			&+ \frac{(-1)^nr^n}{2^n}\sum_{j=0}^\infty \left[ \frac{\psi(1+j)}{2} + \frac{\psi(1+n+j)}{2} -\log\left(\frac{r}{2}\right) \right] \frac{(r^2/4)^j}{j!(n+j)!},
		\end{split}
	\end{equation*}  
where $\psi(\cdot)$ is the digamma function defined by $\psi(z)=\frac{d}{dz}\log \Gamma(z) = \Gamma'(z)/\Gamma(z)$. This implies that, for $n\ge 3$, we have the following expansion as $r\to 0$,
\begin{equation}\label{eq:Taylor2}
	\begin{split}
		\frac{r^n}{2^{n-1}\Gamma(n)}K_n(r) = 1 - \frac{r^2}{4(n-1)} + \frac{r^4}{32(n-1)(n-2)} + o(r^4).
	\end{split}
\end{equation}  
Combining \eqref{eq:Taylor1} and \eqref{eq:Taylor2}, we obtain that, for $\nu>2$,
\begin{equation}\label{eq:Taylor-M}
	\begin{split}
		\mathcal{M}(d) &= \frac{\sigma^2}{2^{\nu-1}\Gamma(\nu)}\left( \frac{\sqrt{2\nu}d}{\ell} \right)^\nu K_\nu\left( \frac{\sqrt{2\nu}d}{\ell} \right)\\
		&=\sigma^2\left[1-\frac{\nu}{2(\nu-1)}\left(\frac{d}{\ell}\right)^2 + \frac{\nu^2}{8(\nu-1)(\nu-2)}\left(\frac{d}{\ell}\right)^4\right] + o(d^4), \quad d\to 0.
	\end{split}
\end{equation}  
In other words, we have
\begin{equation*}
	\begin{split}
		\rho(d) = \mathcal{M}(\sqrt{d}) = \sigma^2\left[1-\frac{\nu}{2(\nu-1)\ell^2}d + \frac{\nu^2}{8(\nu-1)(\nu-2)\ell^4}d^2\right] + o(d^2),  \quad d\to 0.
	\end{split}
\end{equation*}  
This second-order Taylor expansion for $\rho(d)$ around $d=0$ implies \eqref{eq:rho'} and hence \eqref{eq:kappa}.
			\end{proof}
		
		It is well known (see for example \cite{Stein:1999}) that Gaussian fields with covariance function given by \eqref{eq:Matern} are mean-square $m$ times  differentiable if $\nu>m$. Thus, under our assumption $\nu>2$, the Mat\'ern Gaussian fields are mean-square twice differentiable, whose partial derivatives are denoted by
		\[
		X_i(t) = \frac{\partial X(t)}{\partial t_i}, \quad X_{jk}(t) = \frac{\partial^2 X(t)}{\partial t_j\partial t_k}, \quad i,j,k=1, \ldots, N.
		\]
		The following result shows that these derivatives exist almost surely.
		\begin{lemma}\label{lemma:smooth}
			Let $X$ be a Mat\'ern Gaussian field on $\R^N$ with the smoothness parameter $\nu>2$. Then $X$ is twice differentiable almost surely and there exist $c>0$ and $\gamma>0$ such that for any compact set $I$,
			\begin{equation}\label{eq:Lip}
				\E(X_{ij}(t)-X_{ij}(s))^2\le c\|t-s\|^\gamma, \quad \forall t,s\in I.
			\end{equation}
		\end{lemma}
	\begin{proof}
		Denote the covariance function of $X$ by $r(t,s)=\E[X(t)X(s)] = \mathcal{M}(\|t-s\|)$. It then follows from the proof of Proposition \ref{prop:1} that, there exists $\gamma>0$ such that as $\|t-s\|\to 0$,
		\[
		r(t,s)= \sigma^2\left[1-\frac{\nu}{2(\nu-1)}\left(\frac{\|t-s\|}{\ell}\right)^2 + \frac{\nu^2}{8(\nu-1)(\nu-2)}\left(\frac{\|t-s\|}{\ell}\right)^4\right] + o(\|t-s\|^{4+\gamma}).
		\]
		Thus, by writing the covariance of $(X_{ij}(t), X_{ij}(s))$ as partial derivatives of $r(t,s)$ (cf. (5.5.4) in \cite{AT07}) and applying the isotropic property of $X$, we obtain that there exists $c>0$ such that
		\[
		\E(X_{ij}(t)-X_{ij}(s))^2=2\bigg|\frac{\partial^4 r(t,s)}{\partial t_i\partial t_j\partial s_i\partial s_j}\Big|_{t=s} - \frac{\partial^4 r(t,s)}{\partial t_i\partial t_j\partial s_i\partial s_j}\bigg|\le c\|t-s\|^\gamma, \quad \forall t,s\in I.
		\]
		By Theorem 1.4.2 in \cite{AT07}, we obtain that $X$ is twice differentiable almost surely.
	\end{proof}
Note that \eqref{eq:Lip} also implies that $X$ is almost surely a Morse function (cf. Corollary 11.3.2 in \cite{AT07}) and hence we can apply the Kac-Rice formula to compute the expected Euler characteristic and expected number of critical points. 
	
		\subsection{The expected Euler characteristic of excursion set}
		Denote by $\phi(x)=(2\pi)^{-1/2}e^{-x^2/2}$ and $\Psi(x)=\int_x^\infty \phi(y)dy$ the density and tail probability of the standard normal distribution, respectively. Let $H_j(x)$ be the Hermite
		polynomial of order $j$, i.e.,
		$$H_j(x) = (-1)^j e ^{x^2/2} \frac{d^j}{dx^j}\big( e^{-x^2/2} \big), \quad j\ge 0.$$		
		\begin{theorem}\label{thm:EEC}
			Let $\{X(t), t\in T\}$ be a centered Gaussian random field with Mat\'ern covariance function \eqref{eq:Matern}, where $T\subset \R^N$ is an $N$-dimensional piecewise smooth set. Suppose $\nu>2$. Then the expected Euler characteristic is given by
			\begin{equation}\label{eq:EEC}
				\begin{split}
					\E[\chi(A_u(X, T))] = \sum_{j=0}^N  \frac{\nu^{j/2}}{(\nu-1)^{j/2}\ell^j} \mathcal{L}_j (T) \xi_j\left(\frac{u}{\sigma}\right),
				\end{split}
			\end{equation} 
			where $\mathcal{L}_j (T)$ are the Lipschitz-Killing curvatures (cf. (10.7.3) in \cite{AT07}) of $T$ and
			\begin{equation}\label{eq:rhok}
				\begin{split}
					\xi_0(x)= \Psi(x), \quad \xi_j(x) = (2\pi)^{-j/2} H_{j-1} (x) \phi(x), \quad  j\geq 1.
				\end{split}
			\end{equation} 
		\end{theorem}
	\begin{proof} It follows from formula (5.5.5) in \cite{AT07} or Lemma 3.2 in \cite{CS18} that ${\rm Var}[\nabla X(t)]=-2\rho'(0)I_N$, where $I_N$ is the $N\times N$ identity matrix. By \eqref{eq:rho'}, we obtain
		\begin{equation}\label{eq:metric}
		{\rm Var}[\nabla X(t)/\sigma]=\left(\frac{\nu}{(\nu-1)\ell^2}\right)I_N.
	    \end{equation} 
		Applying Theorem 12.4.1 in \cite{AT07} to the standardized Gaussian field $X/\sigma$ yields
		\begin{equation}\label{eq:EEC-Euc}
			\begin{split}
				\E[\chi(A_u(X, T))] = \E[\chi(A_{u/\sigma}(X/\sigma, T))] = \sum_{j=0}^N  \mathcal{L}_j^{X/\sigma} (T) \xi_j\left(\frac{u}{\sigma}\right),
			\end{split}
		\end{equation} 	
	where $\mathcal{L}_j^{X/\sigma} (T)$ are the Lipschitz-Killing curvatures of $T$ calculated with respect to the metric (cf. (12.2.2) in \cite{AT07}) induced by $X/\sigma$. Due to \eqref{eq:metric}, following the arguments on page 324 in \cite{AT07}, we see that the metric induced by $X/\sigma$ makes a new inner product for $t, s \in \R^N$ given by $\frac{\nu}{(\nu-1)\ell^2} \l t, s\r$, where $\l \cdot, \cdot \r$ is the simple Euclidean inner product, implying that 
	\[
	\mathcal{L}_j^{X/\sigma} (T) = \left(\frac{\nu}{(\nu-1)\ell^2}\right)^{j/2}  \mathcal{L}_j (T), \quad j\ge 0.
	\]
	plugging this into \eqref{eq:EEC-Euc} yields the desired result \eqref{eq:EEC}.
		\end{proof}
	
	The formula \eqref{eq:EEC} shows that the expected Euler characteristic is computable with a relatively simple form, mainly due to the isotropy of Mat\'ern Gaussian fields. The Lipschitz-Killing curvatures $\mathcal{L}_j (T)$ depend on the geometry of $T$. We show below an example for the case when $T$ is a cube. 
	\begin{example} Let $T=[0,b]^N$ be an $N$-dimensional cube in $\R^N$. Then, by (10.7.4) in \cite{AT07}, $\mathcal{L}_j (T) = {N\choose j} b^j$, which implies
		\begin{equation*}
			\begin{split}
				\E[\chi(A_u(X, T))] = \phi\left(\frac{u}{\sigma}\right)\sum_{j=1}^N \frac{{N\choose j} b^j \nu^{j/2}}{(2\pi)^{j/2}(\nu-1)^{j/2}\ell^j}H_{j-1}\left(\frac{u}{\sigma}\right) + \Psi\left(\frac{u}{\sigma}\right).
			\end{split}
		\end{equation*} 	
	\end{example}
	
	\subsection{Expected number and height distribution of critical points}
	It is introduced in \cite{CS18} that, an $N\times N$ random matrix $M=(M_{ij})_{1\le i,j\le N}$ is called \emph{Gaussian Orthogonally Invariant} (GOI) with \emph{covariance parameter} $c$, denoted by $\GOI (c)$, if it is symmetric and all entries are centered Gaussian variables such that
	\begin{equation}\label{eq:OI}
		\E[M_{ij}M_{kl}] = \frac{1}{2}(\delta_{ik}\delta_{jl} + \delta_{il}\delta_{jk}) + c\delta_{ij}\delta_{kl},
	\end{equation}
	where $\delta_{ij}$ is the Kronecker delta function. We see that $\GOI(c)$ becomes a GOE matrix if $c=0$. In particular, Lemma 2.2 in \cite{CS18} shows that the density of the ordered eigenvalues $\la_1\le \ldots \le \la_N$ of $\GOI (c)$ is given by
		\begin{equation}\label{Eq:GOI density}
			\begin{split}
				f_c(\la_1, \ldots, \la_N) &=\frac{1}{K_N \sqrt{1+Nc}} \exp\left\{ -\frac{1}{2}\sum_{i=1}^N \la_i^2 + \frac{c}{2(1+Nc)} \left(\sum_{i=1}^N \la_i\right)^2\right\}\\
				&\quad \times  \prod_{1\le i<j \le N} |\la_i - \la_j| \mathbbm{1}_{\{\la_1\leq\ldots\leq\la_N\}},
			\end{split}
		\end{equation}
		where $K_N=2^{N/2}\prod_{i=1}^N\Gamma\left(\frac{i}{2}\right)$ and $c>-1/N$.
	We use the notation $\E_{\GOI(c)}^N$ to represent the expectation under the $\GOI(c)$ density \eqref{Eq:GOI density}, i.e., for a measurable function $g$,
	\begin{equation}\label{Eq:E-GOE}
		\E_{\GOI(c)}^N [g(\la_1, \ldots, \la_N)] = \int_{\R^N} g(\la_1, \ldots, \la_N) f_c(\la_1, \ldots, \la_N) d\la_1\cdots d\la_N.
	\end{equation}
	\begin{theorem} Let $\{X(t), t\in T\}$ be a centered Gaussian random field with Mat\'ern covariance function \eqref{eq:Matern}, where $T\subset \R^N$ is an $N$-dimensional set. Suppose $\nu>2$. Then for $i=0, \ldots, N$,
		\begin{equation}\label{Eq:critial-Euclidean}
			\begin{split}
				\E[\mu_i(X)] &=  \frac{2^{N/2}}{\pi^{N/2}\eta^N} \E_{\GOI(1/2)}^N \left[\prod_{j=1}^N|\la_j|\mathbbm{1}_{\{\la_i<0<\la_{i+1}\}}\right],\\
					\E[\mu_i(X, u)]  &=\frac{2^{N/2}}{\pi^{N/2}\eta^N} \int_{u/\sigma}^\infty \phi(x) \E_{\GOI((1-\kappa^2)/2)}^N \left[\prod_{j=1}^N\big|\la_j-\kappa x/\sqrt{2}\big|\mathbbm{1}_{\{\la_i<\kappa x/\sqrt{2}<\la_{i+1}\}}\right] dx,\\
			F_i(u)  &=\frac{\int_{u/\sigma}^\infty \phi(x) \E_{\GOI((1-\kappa^2)/2)}^N \left[\prod_{j=1}^N\big|\la_j-\kappa x/\sqrt{2}\big|\mathbbm{1}_{\{\la_i<\kappa x/\sqrt{2}<\la_{i+1}\}}\right] dx}{\E_{\GOI(1/2)}^N \left[\prod_{j=1}^N|\la_j|\mathbbm{1}_{\{\la_i<0<\la_{i+1}\}}\right] },
		\end{split}
	\end{equation}
		where $\kappa$ and $\eta$ are given in \eqref{eq:kappa}, $\E_{\GOI(c)}^N$ is defined in \eqref{Eq:GOI density} and $\la_0$ and $\la_{N+1}$ are regarded respectively as $-\infty$ and $\infty$ for consistency. 
	\end{theorem}
	\begin{proof} By the definition \eqref{Eq:mu-i}, we have  $\mu_i(X)=\mu_i(X/\sigma)$ and $\mu_i(X, u)=\mu_i(X/\sigma, u/\sigma)$. Applying Theorem 3.5 in \cite{CS18} to the standardized field $X/\sigma$, we obtain the first and second lines in \eqref{Eq:critial-Euclidean}. Finally, the last line in \eqref{Eq:critial-Euclidean} follows directly from \eqref{Eq:F-ratio}.
	\end{proof}

	\section{Smooth Mat\'ern Gaussian Random Fields on Spheres}\label{sec:2D}
	The applications in geoscience, astronomy and environmental sciences have stimulated recent rapid development in statistics of random fields on spheres. It has been shown in \cite{Gneiting13} that, many of the commonly used covariance functions on Euclidean spaces are valid on spheres when Euclidean distance is replaced by the spherical distance (great circle distance) on a sphere. However, the Mat\'ern class in \eqref{eq:Matern} is positive definite with the spherical distance only if $\nu \le 1/2$.
	
	Let $\S^N$ be an $N$-dimensional unit sphere. Let $\|\cdot\|$ and $\l \cdot, \cdot \r$ be the Euclidean distance and inner product in $\R^{N+1}$, respectively. As shown in \cite{Y83}, one can apply the identity
	\[
	\|x-y\| = 2 \sin \left(\frac{\theta(x, y)} 2\right), \quad \forall x,\, y \in \mathbb S^N \subset \R^{N+1},
	\]
	where $\theta(x,y)= \arccos \l x, y \r \in [0, \pi]$ denotes the spherical distance on $\mathbb{S}^N$, to construct covariance functions on spheres. In particular, by \eqref{eq:Matern}, we define the Mat\'ern covariance function on $\S^N$ as  
	\begin{equation}\label{eq:Matern-sphere}
		\mathcal{M}\left(2\sin(\theta(x,y)/2)\right) = \frac{\sigma^2}{2^{\nu-1}\Gamma(\nu)}\left( \frac{2\sqrt{2\nu}\sin(\theta(x,y)/2)}{\ell} \right)^\nu K_\nu\left( \frac{2\sqrt{2\nu}\sin(\theta(x,y)/2)}{\ell} \right).
	\end{equation} 

    Let $\{X(x), x\in \S^N\}$ be a centered Gaussian random field with Mat\'ern covariance function \eqref{eq:Matern-sphere}. By the construction of the covariance function \eqref{eq:Matern-sphere}, it can be seen that, this Gaussian field on $\S^N$ is simply a Mat\'ern Gaussian field on $\R^{N+1}$ with covariance \eqref{eq:Matern} restricted on $\S^N$. Therefore, one still has the property that the Gaussian field $X$ is twice differentiable almost surely if $\nu>2$. In particular, the smoothness property in Lemma \ref{lemma:smooth} holds for $X_{\varphi_\alpha}=X\circ \varphi_\alpha^{-1}$ on $\varphi_\alpha(U_\alpha)\subset \R^N$, where $\mathcal{A}=(U_\alpha, \varphi_\alpha)_{\alpha\in \mathcal{I}}$ is a finite atlas for $\S^N$. By Corollary 11.3.2 in \cite{AT07}, $X$ is almost surely a Morse function on $\S^N$ and hence we can apply the Kac-Rice formula to compute the expected Euler characteristic and expected number of critical points. 
    
    Note that
	\[
	\sin(\theta(x,y)/2) = \sqrt{\frac{1-\cos\theta(x,y)}{2}} = \sqrt{\frac{1-\l x, y\r}{2}}.
	\]
	Following the notation in \cite{CS17}, we write the covariance of $X$ as
	\begin{equation}\label{eq:Matern-sphere2}
		\begin{split}
		C(\l x, y\r)&:= \E[X(x)X(y)] = \mathcal{M}\left(\sqrt{2(1-\l x, y\r)}\right) \\
		&=  \frac{\sigma^2}{2^{\nu-1}\Gamma(\nu)}\left( \frac{\sqrt{2\nu}\sqrt{2(1-\l x, y\r)}}{\ell} \right)^\nu K_\nu\left( \frac{\sqrt{2\nu}\sqrt{2(1-\l x, y\r)}}{\ell} \right), \ \forall x,\, y \in \mathbb S^N,
		\end{split}
	\end{equation}
and derive the following results.
\begin{proposition}
	Let $\{X(x), x\in \S^N\}$ be a centered Gaussian random field with Mat\'ern covariance function \eqref{eq:Matern-sphere2}. Suppose $\nu>2$. Let $C(p):=C(\l x, y\r)$, $p\in [-1,1]$; and let $C'=C'(1)$ and $C''=C''(1)$. Then
	\begin{equation}\label{eq:C'}
		\begin{split}
			C'= \frac{\sigma^2\nu}{(\nu-1)\ell^2} \quad  \text{\rm and } \quad 
			C''= \frac{\sigma^2\nu^2}{(\nu-1)(\nu-2)\ell^4}.
		\end{split}
	\end{equation} 
In particular,
\begin{equation}\label{eq:eta}
	\begin{split}
		\tilde{\kappa} := \frac{C'}{\sqrt{C''}}\bigg\rvert_{\sigma=1} = \sqrt{\frac{\nu-2}{\nu-1}} \quad  \text{\rm and } \quad \tilde{\eta} :=\frac{\sqrt{C'}}{\sqrt{C''}} = \frac{\sqrt{\nu-2}}{\sqrt{\nu}}\ell.
	\end{split}
\end{equation} 
\end{proposition}
\begin{proof}
	By \eqref{eq:Taylor-M}, we obtain that, as $p\to 1$,
	\begin{equation}\label{eq:Taylor-MS}
		\begin{split}
			C(p)&=\mathcal{M}\left(\sqrt{2(1-p)}\right) \\
			&=\sigma^2\left[1-\frac{\nu}{(\nu-1)\ell^2}(1-p) + \frac{\nu^2}{2(\nu-1)(\nu-2)\ell^4}(1-p)^2\right] + o\left((1-p)^2\right)\\
			&=\sigma^2\left[1+\frac{\nu}{(\nu-1)\ell^2}(p-1) + \frac{\nu^2}{2(\nu-1)(\nu-2)\ell^4}(p-1)^2\right] + o\left((p-1)^2\right).
		\end{split}
	\end{equation}  
This second-order Taylor expansion for $C(p)$ around $p=1$ implies \eqref{eq:C'} and hence \eqref{eq:eta}.
\end{proof}

Let $\omega_j =   \frac{2\pi^{(j+1)/2}}{\Gamma((j+1)/{2})}$ be the spherical area of the $j$-dimensional unit sphere $\mathbb{S}^j$. We have the following results on the expected Euler characteristic of the excursion set $A_u(X, \S^N)=\{x\in \S^N: X(x) \ge u\}$, as well as the expected number and height distribution of critical points.
	\begin{theorem}
	Let $\{X(x), x\in \S^N\}$ be a centered Gaussian random field with Mat\'ern covariance function \eqref{eq:Matern-sphere2}. Suppose $\nu>2$. Then the expected Euler characteristic is 
	\begin{equation}\label{eq:EECS}
		\begin{split}
			\E[\chi(A_u(X, \S^N))] = \sum_{j=0}^N  \frac{\nu^{j/2}}{(\nu-1)^{j/2}\ell^j}  \mathcal{L}_j (\mathbb{S}^N) \xi_j\left(\frac{u}{\sigma}\right),
		\end{split}
	\end{equation} 
where $\xi_j(\cdot)$ are given in \eqref{eq:rhok} and
\begin{equation}\label{Eq:L-K curvature}
	\begin{split}
		\mathcal{L}_j (\mathbb{S}^N) = \left\{
		\begin{array}{l l}
			2 \binom{N}{j}\frac{\omega_N}{\omega_{N-j}} & \quad \text{if $N-j$ is even,}\\
			0 & \quad \text{otherwise},
		\end{array} \right. \quad j=0, 1, \ldots, N,
	\end{split}
\end{equation}
are the Lipschitz-Killing curvatures of $\mathbb{S}^N$.
\end{theorem}
\begin{proof} Applying Lemma 3.5 in \cite{ChengXiao16} to the standardized Gaussian field $X/\sigma$ yields
	\begin{equation*}
		\begin{split}
			\E[\chi(A_u(X, \S^N))] &= \E[\chi(A_{u/\sigma}(X/\sigma, \S^N))] = \sum_{j=0}^N  \left(\frac{C'}{\sigma^2}\right)^{j/2} \mathcal{L}_j (\mathbb{S}^N) \xi_j\left(\frac{u}{\sigma}\right)\\
			&= \sum_{j=0}^N  \left(\frac{\nu}{(\nu-1)\ell^2}\right)^{j/2} \mathcal{L}_j (\mathbb{S}^N) \xi_j\left(\frac{u}{\sigma}\right),
		\end{split}
	\end{equation*} 	
where $C'$ is given in \eqref{eq:C'}.
\end{proof}

Note that, Lemma 3.5 in \cite{ChengXiao16} is a result for smooth Gaussian fields on spheres with covariance functions having certain Taylor expansion. The essential requirement is that the Gaussian field is twice differentiable. Here, $X$ is a smooth ($\nu>2$) Mat\'ern Gaussian field on $\R^{N+1}$ restricted on $\mathbb{S}^N$, following the arguments subsequent to \eqref{eq:Matern-sphere}, we see that the smoothness requirement can be guaranteed by Corollary 11.3.2 in \cite{AT07} and hence Lemma 3.5 in \cite{ChengXiao16} is applicable.

\begin{theorem}
	Let $\{X(x), x\in \S^N\}$ be a centered Gaussian random field with Mat\'ern covariance function \eqref{eq:Matern-sphere2}. Suppose $\nu>2$. Then for $i=0, \ldots, N$,
	\begin{equation}\label{Eq:critial-sphere}
		\begin{split}
			\E[\mu_i(X)] &=  \frac{1}{\pi^{N/2}\tilde{\eta}^N} \E_{\GOI((1+\tilde{\eta}^2)/2)}^N \left[\prod_{j=1}^N|\la_j|\mathbbm{1}_{\{\la_i<0<\la_{i+1}\}}\right],\\
			\E[\mu_i(X, u)]  &=\frac{1}{\pi^{N/2}\tilde{\eta}^N} \int_{u/\sigma}^\infty \phi(x) \E_{\GOI((1+\tilde{\eta}^2-\tilde{\kappa}^2)/2)}^N \left[\prod_{j=1}^N\big|\la_j-\tilde{\kappa} x/\sqrt{2}\big|\mathbbm{1}_{\{\la_i<\tilde{\kappa} x/\sqrt{2}<\la_{i+1}\}}\right] dx,\\
			F_i(u)  &=\frac{\int_{u/\sigma}^\infty \phi(x) \E_{\GOI((1+\tilde{\eta}^2-\tilde{\kappa}^2)/2)}^N \left[\prod_{j=1}^N\big|\la_j-\tilde{\kappa} x/\sqrt{2}\big|\mathbbm{1}_{\{\la_i<\tilde{\kappa} x/\sqrt{2}<\la_{i+1}\}}\right] dx}{\E_{\GOI((1+\tilde{\eta}^2)/2)}^N \left[\prod_{j=1}^N|\la_j|\mathbbm{1}_{\{\la_i<0<\la_{i+1}\}}\right] },
		\end{split}
	\end{equation}
	where $\tilde{\kappa}$ and $\tilde{\eta}$ are given in \eqref{eq:eta},  $\E_{\GOI(c)}^N$ is defined in \eqref{Eq:GOI density} and $\la_0$ and $\la_{N+1}$ are regarded respectively as $-\infty$ and $\infty$ for consistency. 
\end{theorem}
\begin{proof} Applying Theorem 4.4 in \cite{CS18} to the standardized field $X/\sigma$, we obtain the first and second lines in \eqref{Eq:critial-sphere}. The last line in \eqref{Eq:critial-sphere} is a direct consequence of \eqref{Eq:F-ratio}.
\end{proof}

\section*{Acknowledgments}

The author thanks Yimin Xiao for introducing this research problem. He is also grateful to the editor and the anonymous referee for their efforts and valuable comments, which have led to great improvement for the paper.

	\begin{small}
		
	\end{small}
	
	\bigskip
	
	\begin{quote}
		\begin{small}
			
			\textsc{Dan Cheng}\\
			School of Mathematical and Statistical Sciences \\
			Arizona State University\\
			900 S. Palm Walk\\
			Tempe, AZ 85281, U.S.A.\\
			E-mail: \texttt{cheng.stats@gmail.com}

		\end{small}
	\end{quote}
	

\begin{thebibliography}{9}
			
			\bibitem{AT07}
			Adler, R. J. and Taylor, J. E. (2007). {\it Random Fields and Geometry}. Springer, New York.
			
			\bibitem{AstrophysJ85}
			Bardeen, J. M., Bond, J. R., Kaiser, N. and Szalay A. S. (1985). The statistics of peaks of Gaussian random fields. {\it Astrophys. J.},
			{\bf 304},  15--61.
			
			\bibitem{CS19sphere}
			Cheng, D., Cammarota, V., Fantaye, Y., Marinucci, D. and Schwartzman, A. (2020). Multiple testing of local maxima for detection of peaks on the (celestial) sphere. {\it Bernoulli}, {\bf 26},  31--60.
			
			\bibitem{ChangePoint}
			Cheng, D., He, Z. and Schwartzman, A. (2020). Multiple testing of local extrema for detection of change points. {\it Electron. J. Statist.}, {\bf 14},  3705--3729.
			
			\bibitem{CS15}
			Cheng, D. and Schwartzman, A. (2015). Distribution of the height of local maxima of Gaussian random fields. {\it Extremes}, {\bf 18}, 213--240.
			
			\bibitem{CS18}
			Cheng, D. and Schwartzman, A. (2018). Expected number and height distribution of critical points of smooth isotropic Gaussian random fields. {\it Bernoulli}, {\bf 24}, 3422--3446.
			
			\bibitem{CS17}
			Cheng, D. and Schwartzman, A. (2017). Multiple testing of local maxima for detection of peaks in random fields. {\it Ann. Stat.}, {\bf 45}, 529--556.
			
			\bibitem{ChengXiao16}
			Cheng, D. and Xiao, Y. (2016), Excursion probability of Gaussian random fields on sphere. {\it Bernoulli}, {\bf 22}, 1113--1130.
			
			\bibitem{Chumbley:2009}
			Chumbley, J. R and Friston, K. J. (2009). False discovery rate revisited: FDR and topological inference using
			Gaussian random fields. \emph{Neuroimage}, {\bf 44}, 62--70.
			
			\bibitem{Gneiting13}
			Gneiting, T. (2013), Strictly and non-strictly positive definite functions on spheres.
			{\it Bernoulli} {\bf 19}, 1327--1349.
			
			
			\bibitem{Guttorp:2006}
			Guttorp, P. and Gneiting, T. (2006). Studies in the history of probability and statistics XLIX: On the Mat\'ern correlation family. {\it Biometrika}. {\bf 93}, 989--995.
			
			\bibitem{Spanier:1987}
			Spanier, J. and Oldham, K. (1987). {\it An Atlas of Functions}. Springer.
			
			\bibitem{Stein:1999}
			Stein, M. (1999). {\it Interpolation of Spatial Data: Some Theory for Kriging}.
			Springer.
			
			\bibitem{Taylor:2007}
			Taylor, J. E. and Worsley, K. J. (2007). Detecting sparse signals in random fields, with an application to brain mapping. \emph{J. Am. Statist. Assoc.}, {\bf 102}, 913--928.
			
			\bibitem{Y83}
			Yadrenko, A. M. (1983). {\it Spectral Theory of Random Fields.} Optimization Software, New York.
			
		\end{thebibliography}
\end{document}